\documentclass[11pt]{amsart}
\usepackage[margin=1.1in]{geometry}
\usepackage{amsmath,amssymb,amsthm,mathtools}
\usepackage[dvipsnames]{xcolor}
\usepackage[colorlinks=true,linkcolor=blue!50!black,citecolor=blue!50!black,urlcolor=blue!50!black]{hyperref}
\usepackage{enumitem}

\newtheorem{theorem}{Theorem}[section]
\newtheorem{lemma}[theorem]{Lemma}
\newtheorem{proposition}[theorem]{Proposition}
\newtheorem{corollary}[theorem]{Corollary}
\theoremstyle{definition}

\theoremstyle{remark}

\newcommand{\Z}{\mathbb{Z}}
\newcommand{\Q}{\mathbb{Q}}
\newcommand{\nint}[1]{\lVert #1\rVert}

\title[Unit fractions with semiprime denominators]{Unit fractions with semiprime denominators:\\ an elementary proof of Erd\H{o}s Problem \#306}
\author{Shisheng Li}
\email{shisheng@mail.ustc.edu.cn}
\subjclass[2020]{Primary 11D68; Secondary 11B75, 11N05}
\keywords{Egyptian fractions, squarefree semiprimes, finite Fourier analysis, Chinese remainder theorem}

\begin{document}

\begin{abstract}
We give an elementary proof that every positive rational number $a/b$ with $b$ squarefree is a finite sum of distinct unit fractions $1/n$, where each $n$ is a product of two distinct primes (Erd\H{o}s Problem \#306). After a reduction to small targets, we take a single complete bipartite graph between the primes in $(y^2,2y^2]$, together with $2$ and the primes of $b$, and a tuned initial segment of the primes in $(y^8,y^9]$, and show that some subgraph has reciprocal sum congruent to $a/b$ modulo $1$; the small total mass then forces equality. Writing the number of such subgraphs as a finite Fourier sum, we sort the frequencies into three cases using a table indexed by the two sides of the graph. The small integer frequencies give a positive main term, and all other frequencies are negligible by a divisor-counting argument and a no-wrap-around form of the Chinese remainder theorem. The only inputs about primes are Chebyshev-type bounds. The circle-method framework comes from Tang's Lean development, which gave the first proof; our construction removes its anchor-synchronisation step. The proof has been formalised in Lean~4, apart from a cited inequality of Ramanujan. This work is a human--AI collaboration: AI tools contributed substantially to the construction, the experiments and the writing.
\end{abstract}

\maketitle

\section{Introduction}

Erd\H{o}s and Graham \cite{ErGr80} asked the following question (Erd\H{o}s Problem \#306 in \cite{Bloom}): given $a/b\in\Q_{>0}$ with $b$ squarefree, are there integers $1<n_1<\dots<n_k$, each the product of two distinct primes, such that
\[
  \frac ab=\frac1{n_1}+\dots+\frac1{n_k}\ ?
\]
The condition on $b$ (for $a/b$ in lowest terms) is necessary, since a sum of reciprocals of squarefree numbers has squarefree reduced denominator. For products of three distinct primes, the analogous statement was proved for integers by Butler, Erd\H{o}s and Graham \cite{BEG15}. For two primes, Li \cite{Li26} proved it for all $a/b$ above an explicit threshold $T(r)\le1/5$ depending on the largest prime factor $r$ of $b$. The remaining difficulty lies with small values of $a/b$ whose denominator has large prime factors.

\begin{theorem}\label{thm:main}
Let $a/b\in\Q_{>0}$ with $b$ squarefree. Then there are distinct integers $1<n_1<\dots<n_k$, each a product of two distinct primes, with $a/b=\sum_{i}1/n_i$.
\end{theorem}

The proof is elementary in the usual sense of the word in number theory. The only input about primes consists of Chebyshev-type bounds, namely Erd\H{o}s' bound $\prod_{p\le x}p\le4^x$ and Ramanujan's lower bound for $\theta(x)-\theta(x/2)$; neither the prime number theorem nor any complex analysis is used. The only Fourier-analytic tool is the geometric-series identity $\frac1L\sum_{h\bmod L}e(hm/L)=\mathbf 1_{L\mid m}$, a finite sum. All constants are explicit, and every threshold ``$y$ sufficiently large'' is effective, although we do not compute it.

\subsection*{Formal verification}\label{sec:lean}
The proof has been formalised in Lean~4 with Mathlib; the source files (about $3\cdot10^3$ lines) are included as ancillary files of this arXiv submission, in the directory \texttt{anc/lean}. The formalisation follows the paper step by step, with one Lean declaration for each numbered statement, whose documentation cites that number; it is a check of the present proof, not an independent proof of Theorem~\ref{thm:main}. Every condition ``$y$ sufficiently large'' is stated there as an explicit inequality in $y$, and the existence of $y$ satisfying all of them is proved. The only statement assumed without proof is Ramanujan's inequality $\theta(x)-\theta(x/2)>x/6-3\sqrt x$ for $x>300$ \cite{Ram19}, used in Lemma~\ref{lem:cheb}. It enters as an explicit hypothesis of the main Lean theorem \texttt{erdos\_306\_of\_ramanujan}, which contains no \texttt{sorry} and depends only on the standard axioms \texttt{propext}, \texttt{Classical.choice} and \texttt{Quot.sound}.

\subsection*{Outline}
A finite set of squarefree semiprimes is a finite simple graph on the primes: the edge $\{p,q\}$ stands for $pq$. Splitting $a/b$ into $q$ equal parts, for a large prime $q$, reduces the theorem to targets $a/b\le\eta$, where $\eta$ is a fixed small absolute constant (\S\ref{sec:reduction}).

For such a target, the construction is a single complete bipartite graph. On one side is $V=\{2\}\cup B\cup W$, where $B$ is the set of primes of $b$ and $W$ is the set of primes in $(y^2,2y^2]$. On the other side is the set $U$ of primes in $(y^8,Y]$, with $y^8<Y\le y^9$ chosen so that half the total mass of the graph is just below $a/b$. Every subgraph then has total mass at most $\frac12$, so a congruence modulo $1$ with $a/b$ already forces equality. It therefore suffices to show that some subgraph $S$ satisfies $\sum_{n\in S}1/n\equiv a/b\pmod1$. Let $L$ be the product of all the primes of the graph. A geometric sum over the integers $-L/2<h\le L/2$ writes the number of such subgraphs as an average over these \emph{frequencies} $h$. By the Chinese remainder theorem, a frequency is the same thing as a residue at every prime of the graph, and the absolute value $|\phi(h)|$ of the corresponding Fourier coefficient is the product of the entries of a $V\times U$ table, whose $(v,u)$ entry depends only on the residue of $h$ at $v$ (the \emph{row label}) and at $u$ (the \emph{column label}). We sort all frequencies by two questions: do the row labels all come from one small integer $M$ (\emph{coherent} rows), and if so, are all column labels congruent to $M$ as well? This gives three cases (\S\ref{sec:table}).
\begin{enumerate}[label=\Roman*.]
\item \emph{Coherent rows, all columns matching.} These are exactly the small frequencies $h=M$, $|M|\le y^7$: the major arcs (\S\ref{sec:major}). Because half the total mass of the graph is very close to $a/b$ and all the angles $M/(uv)$ are small, the contributions of $M$ and $-M$ together are nonnegative, and $h=0$ gives the main term.
\item \emph{Coherent rows, some mismatched column.} The many test primes in $W$ make a mismatched label conspicuous: by a divisor count, a mismatched column has weight at most $\exp(-y^2/(2\cdot10^7\log^3y))$ (Lemmas~\ref{lem:count} and~\ref{lem:unique}). There are exponentially many sets of mismatched columns, but what matters is their weighted sum: over all $M$ and all mismatch patterns it is at most $3y^{25}\varepsilon e^{y^{18}\varepsilon}$, where $\varepsilon$ is bounded by the exponential just displayed, and this tends to $0$.
\item \emph{Incoherent rows.} If some label made all entries of its column close to $1$, the phase numerators would be too small to differ by a nonzero multiple of $u$, and the rows would be coherent (Lemma~\ref{lem:admissible}). So every column has an entry that is not extremely close to $1$, and by the same divisor count at most one label per column matters. Writing $\tau=a/b$ and $H_V=\sum_{v\in V}1/v$, each of the $|U|\ge\tau y^8/H_V$ columns therefore loses a factor $e^{-\delta_a^2}$, $\delta_a=1/(8y^2)$, and the total loss $\exp(-\tau y^4/(64H_V))$ beats the number $\prod_{v\in V}v\le2b\cdot16^{y^2}$ of row patterns.
\end{enumerate}
Cases~II and~III are the minor arcs (\S\ref{sec:minor}).

The particular exponents are not important: with the small side at scale $y^2$ and the large side in $(y^{2k},y^{2k+2\lambda}]$, the same argument works for any fixed $k>3$ and $0<\lambda<1$, after adjusting, in terms of $k$ and $\lambda$, the constants $c_U$, $\eta$ and $\delta$, the number of large prime factors in the divisor count, and the major-arc cutoff; we take $k=4$ and $\lambda=\frac12$.

\subsection*{Relation to previous work}\label{sec:intro-relation}

\emph{Our earlier paper \cite{Li26}.} In \cite{Li26} we adapted the framework of Butler, Erd\H{o}s and Graham \cite{BEG15} to products of two primes. There the problem is reduced to showing that the subset sums of the numbers $P_N/(p_ip_j)$, where $P_N$ is the $N$-th primorial, cover a central interval. That covering is proved by induction on $N$, using Olson's theorem \cite{Ols68} in $\Z/p_{N+1}\Z$ at each step. This gives all natural numbers, and all $a/b\ge T(r)$ with $T(r)\le1/5$. That covering is proved only above the threshold $T(r)$, and the method in its present form does not reach smaller targets whose denominators have large prime factors.

The present paper does not use the covering induction or Olson's theorem. The proof here is a different argument: finite Fourier analysis on a single complete bipartite graph between two scales of primes. It treats all targets uniformly, and Theorem~\ref{thm:main} contains the main result of \cite{Li26}.

\emph{The Lean development of Tang \cite{Tan26}.} In 2026 Tang released a Lean~4 development proving Theorem~\ref{thm:main} (version 0.0.3 is archived on Zenodo); it is complete except for two stated axioms, transcribing prime estimates of Rosser and Schoenfeld \cite{RS62}, on which it is conditional. It was the first proof of Theorem~\ref{thm:main}, and the present paper would not exist without it. After \cite{Li26} we had tried a circle-method approach ourselves, without success; it was reading the source code of \cite{Tan26} that showed us that the circle method does work for this problem, and the framework of the present proof is taken from there. We have not independently re-verified the Lean development. We now describe its framework, what we adopted, how our argument differs, and what we regard as new.

\emph{The framework of \cite{Tan26}.} At the level of ideas, \cite{Tan26} proceeds as follows. The target is encoded, by finite Fourier inversion modulo $L$, as a positive Fourier coefficient of a random sum of reciprocals of semiprimes, whose mean is matched to the target. The semiprimes are built from blocks of primes, with edges inside blocks, and the primes of $b$ are attached to them. After the Chinese remainder theorem, the argument has three layers.
\begin{itemize}
\item \emph{Anchor synchronisation.} A rigidity statement shows that, for frequencies of low energy, the coordinates of an \emph{anchor} set of primes, which carries internal edges, all come from one integer label.
\item \emph{Decoding.} The remaining primes are decoded against the anchor, and the residues that are not decoded are removed by a product estimate.
\item \emph{Main term.} The small integer frequencies give a positive Gaussian main term, and all other frequencies are shown to be negligible. A total reciprocal load below $1$ then turns the congruence into equality.
\end{itemize}

\emph{What we adopted from \cite{Tan26}.}
\begin{itemize}
\item \emph{The Fourier framework:} finite Fourier inversion modulo $L$ in Chinese remainder coordinates, with phases $h/(pq)$, and a quadratic energy controlled by a Gaussian-type bound for each factor in terms of the distance of its phase from the nearest integer (for us, Lemma~\ref{lem:cos}).
\item \emph{The decoding structure:} the coordinates of one set of primes are forced to come from a single integer, and every prime joined to that set is decoded against it. This includes the observation that the phase change caused by changing one residue does not depend on the anchor (our \eqref{eq:pf}), and the triangle-inequality argument showing that at most one residue can have small energy (the idea behind our Lemma~\ref{lem:unique}).
\item \emph{The handling of the target:} attaching the primes of $b$, matching the mean, taking the small integer frequencies as the main term, obtaining exact equality from a total mass below $1$, and reducing $a/b$ to $p$ disjoint copies of $a/(bp)$.
\end{itemize}

\emph{How our argument differs.}
\begin{itemize}
\item \emph{The graph.} \cite{Tan26} uses blocks of primes with edges inside blocks, plus edges attaching the primes of $b$. We use one complete bipartite graph between two widely separated scales, $V$ up to $2y^2$ and $U$ above $y^8$, with no edges inside either side.
\item \emph{The anchor and coherence.} In \cite{Tan26} the anchor carries internal edges, and synchronising it through its low internal energy is the central step. Our small side $V$ plays the role of the anchor, but it is never synchronised: we count all its residue patterns (at most $2b\cdot16^{y^2}$) and pay for them with a small loss in each column, and a single well-aligned column already forces coherence, because small phase numerators cannot wrap around (Lemma~\ref{lem:admissible}).
\item \emph{Separation.} In \cite{Tan26} the residues of a decoded prime are separated using the anchor primes. Here the labels of a large prime $u$ are separated using test primes $w$ much smaller than $u$, by a divisor count (Lemma~\ref{lem:count}).
\item \emph{Major arcs.} \cite{Tan26} proves a Gaussian main term. Here the mean is matched so closely that every pair of major-arc terms is nonnegative, and positivity alone suffices.
\item \emph{Targets and weights.} \cite{Tan26} uses Bernoulli weights adapted to the target. We treat small targets $\tau\le1/400$ and count all subsets equally (probability $\frac12$ in that language), matching the mean by moving one endpoint (Lemma~\ref{lem:tuning}).
\end{itemize}

\emph{What is new here, and its advantages.}
\begin{itemize}
\item \emph{No synchronisation of an internally connected anchor.} The two-scale bipartite design replaces anchor synchronisation, the technical core of \cite{Tan26}, by the no-wrap-around lemma for a single column (Lemma~\ref{lem:admissible}). The separation of residues then needs only a divisor count (Lemma~\ref{lem:count}), which works with test primes much smaller than the column prime.
\item \emph{A simple classification.} All frequencies fall into three cases, determined by two questions about the $V\times U$ table (\S\ref{sec:table}), instead of a decomposition into several frequency ranges.
\item \emph{No asymptotic analysis on the major arcs.} No Gaussian approximation, Taylor expansion or variance computation is needed.
\item \emph{Weaker inputs.} We use only Chebyshev-type bounds (Erd\H{o}s' $\prod_{p\le x}p\le4^x$ and Ramanujan's bound for $\theta(x)-\theta(x/2)$), with all constants explicit, instead of the Rosser--Schoenfeld estimates.
\item \emph{Checkable by hand.} The argument is short enough to be checked by hand. It has also been formalised in Lean (about $3\cdot10^3$ lines), apart from one cited inequality of Ramanujan (see ``Formal verification'' above).
\end{itemize}

\section{Two bounds for primes}\label{sec:inputs}

The letters $p,q,r,u,v,w$ denote primes, and $\log$ is the natural logarithm. We need two Chebyshev-type bounds.

\begin{lemma}[Chebyshev-type bounds]\label{lem:cheb}
For all $x\ge2$ we have $\prod_{p\le x}p\le 4^x$. For all sufficiently large $x$,
\[
 \frac{x}{7\log x}\ \le\ \pi(x)-\pi(x/2).
\]
\end{lemma}
\begin{proof}
The first bound is Erd\H{o}s' classical inequality \cite{Erd32}. For the second, Ramanujan \cite{Ram19} proved $\theta(x)-\theta(x/2)>x/6-3\sqrt x$ for $x>300$, where $\theta(x)=\sum_{p\le x}\log p$. Each prime in $(x/2,x]$ contributes at most $\log x$ to $\theta(x)-\theta(x/2)$, so $\pi(x)-\pi(x/2)\ge(x/6-3\sqrt x)/\log x$, and this is at least $x/(7\log x)$ for large $x$.
\end{proof}

The first bound will control the number of residue patterns on the small side; the second counts primes in dyadic intervals. The following consequence is used only once, in the tuning lemma, to guarantee that the large side has enough reciprocal mass.

\begin{corollary}\label{cor:mertens}
For all sufficiently large $y$,
\[
 \sum_{y^8<p\le y^9}\frac1p\ \ge\ c_U:=\frac1{50} .
\]
\end{corollary}
\begin{proof}
By Lemma~\ref{lem:cheb}, for large $j$ the dyadic block $(2^{j-1},2^j]$ contains at least $2^j/(7j\log2)$ primes, each with $1/p\ge2^{-j}$, so the block contributes at least $1/(7j\log 2)$ to the sum. The blocks lying inside $(y^8,y^9]$ are those with $8\log_2y+1\le j\le9\log_2y$, and $\sum1/j$ over this range of $j$ tends to $\log\frac98$ as $y\to\infty$. Since $\log\frac98/(7\log2)>0.024>\frac1{50}$, the bound follows for large $y$.
\end{proof}

\section{Reduction to small targets}\label{sec:reduction}

The construction of \S\ref{sec:construction} only works for small targets: we need the total mass of the graph to be at most $\frac12$, so that a congruence modulo $1$ forces equality, and we need the large side to carry enough reciprocal mass to reach the target. With $c_U$ as in Corollary~\ref{cor:mertens}, fix
\[
 \eta=\min\big\{\tfrac14,\ \tfrac18c_U\big\}=\frac1{400}.
\]

\begin{theorem}[Small targets]\label{thm:small}
Let $\tau=a/b\in(0,\eta]$ with $b$ squarefree. Then for every sufficiently large integer $y$ (in terms of $b$), $\tau$ is a finite sum of distinct $1/(vu)$ with $v\le 2y^2<y^8<u\le y^9$ primes.
\end{theorem}

Theorem~\ref{thm:small} is proved in \S\S\ref{sec:construction}--\ref{sec:minor}. The freedom in $y$ is what allows us to add up several representations without collisions.

\begin{proof}[Proof of Theorem~\ref{thm:main} from Theorem~\ref{thm:small}]
Let $a/b>0$ with $b$ squarefree. We may assume $a/b$ is in lowest terms, since reducing it only shrinks the squarefree denominator. There are infinitely many primes, so we can choose a prime $q\nmid ab$ so large that $a/(bq)\le\eta$. Since $q\nmid a$ and $q\nmid b$, the fraction $x=a/(bq)$ is again in lowest terms, and its denominator $bq$ is squarefree.

Now choose $y_1<y_2<\dots<y_q$, each large enough for Theorem~\ref{thm:small} applied to $x$, and with $y_{i+1}^8>y_i^9$; then the intervals $(y_i^8,y_i^9]$ are pairwise disjoint. For each $i$, Theorem~\ref{thm:small} with $y=y_i$ gives a representation of $x$ by distinct semiprimes $vu$ whose larger prime factor $u$ lies in $(y_i^8,y_i^9]$. A semiprime used in the $i$-th representation has its larger prime factor in the $i$-th interval, so it cannot occur in any other representation. Hence the $q$ representations together use pairwise distinct semiprimes, and their sum is $qx=a/b$.
\end{proof}

\section{The construction}\label{sec:construction}

Fix $\tau=a/b\in(0,\eta]$ with $b$ squarefree. Replacing $a/b$ by its reduced form keeps $b$ squarefree, so we may assume $\gcd(a,b)=1$. Let $B$ be the set of primes dividing $b$. Since $0<\tau<1$ we have $b>1$, so $B\ne\emptyset$. Let $y>\max B$ be an integer, large in terms of $b$ (since $\tau\ge1/b$, thresholds depending on $\tau$ also depend only on $b$), and put
\begin{gather*}
 W=\{w:\ y^2<w\le 2y^2\},\qquad V=\{2\}\cup B\cup W;\\
 U=U(Y)=\{u:\ y^8<u\le Y\},\qquad y^8<Y\le y^9.
\end{gather*}
The graph is the complete bipartite graph between the \emph{small side} $V$ and the \emph{large side} $U$. Each part of $V$ has its own job. The primes of $B$ are there so that the denominator $b$ divides the product of all primes of the graph. The prime $2$ gives the small side a fixed amount of reciprocal mass, and later it pins down the size of an integer (Lemma~\ref{lem:admissible}). The primes of $W$ are the \emph{test primes}: there are many of them, and they distinguish residues at a large prime (Lemmas~\ref{lem:count} and~\ref{lem:unique}).

Let $H_V=\sum_{v\in V}1/v$ be the sum of the reciprocals of the small side. For $y$ large we have
\begin{equation}\label{eq:basic}
 |W|\ge\frac{y^2}{8\log y},\qquad \prod_{v\in V}v\le 2b\cdot16^{y^2},\qquad
 \tfrac12\le H_V\le\tfrac32+\sum_{r\in B}\frac1r,\qquad \sum_{u\in U(y^9)}\frac1u\ge c_U.
\end{equation}
Indeed, Lemma~\ref{lem:cheb} with $x=2y^2$ gives $|W|\ge2y^2/(7\log 2y^2)\ge y^2/(8\log y)$. The primes of $B$ multiply to $b$, and $W$ consists of primes up to $2y^2$, so $\prod_Vv\le2\cdot b\cdot4^{2y^2}$ by Lemma~\ref{lem:cheb}. Next, $H_V\ge\frac12$ because $2\in V$, and $H_V\le\frac12+\sum_{r\in B}1/r+1$ because $\sum_{w\in W}1/w\le|W|/y^2\le1$. The last bound is Corollary~\ref{cor:mertens}.

Let $A=A(Y)=\{vu:\ v\in V,\ u\in U\}$ be the set of edges. Since $V\cap U=\emptyset$, each edge $n=vu$ determines the pair $(v,u)$, and we index $A$ by these pairs. Put
\[
 \mu(Y)=\frac12\sum_{(v,u)\in A}\frac1{uv}=\frac12H_V\sum_{u\in U(Y)}\frac1u ,
\]
that is, half the total mass of $A$. The next lemma chooses the endpoint $Y$ so that $\mu$ is just below $\tau$.

\begin{lemma}[Tuning]\label{lem:tuning}
There is a prime $Y\in(y^8,y^9]$ with
\[
 0\le\tau-\mu(Y)\le\frac{H_V}{2y^8},\qquad \mu(Y)\ge\frac\tau2,\qquad |U|\ge\frac{\tau y^8}{H_V}.
\]
\end{lemma}
\begin{proof}
Add the primes of $(y^8,y^9]$ to $U$ one at a time, in increasing order. Adding $u$ increases $\mu$ by $H_V/(2u)$, which is less than $H_V/(2y^8)$, and this is less than $\tau/2$ for $y$ large. With only the first prime, $\mu<\tau$. With all of them, $\mu\ge\frac12\cdot\frac12c_U\ge2\eta\ge2\tau$ by \eqref{eq:basic} and the choice of $\eta$. So we can stop at the last prime $Y$ with $\mu(Y)\le\tau$. Adding the next prime would push $\mu$ above $\tau$, so $\tau-\mu(Y)<H_V/(2y^8)<\tau/2$; in particular $\mu(Y)\ge\tau/2$. Finally, $\sum_{u\in U}1/u=2\mu/H_V\ge\tau/H_V$, and every $u$ exceeds $y^8$, so $|U|\ge\tau y^8/H_V$.
\end{proof}

From now on $Y$ is as in Lemma~\ref{lem:tuning}, and we write $\beta=\mu-\tau$, so that $-H_V/(2y^8)\le\beta\le0$.

The total mass of the graph is $2\mu\le2\tau\le\frac12$, so every $S\subseteq A$ satisfies
\[
 0\le\sum_{n\in S}\frac1n\le\tfrac12 .
\]
Suppose $\sum_{n\in S}1/n\equiv\tau\pmod1$. The difference $\sum_{n\in S}1/n-\tau$ is then an integer, and it lies in $[-\tau,\frac12-\tau]\subset(-1,1)$, so it is $0$. Thus a congruence modulo $1$ already gives an exact representation of $\tau$. The edges $vu$ are distinct semiprimes with $v\le 2y^2<y^8<u\le y^9$. So Theorem~\ref{thm:small} follows from:

\begin{proposition}\label{prop:core}
There is $S\subseteq A$ with $\sum_{n\in S}1/n\equiv\tau\pmod1$.
\end{proposition}

\section{Counting with a geometric sum}\label{sec:count}

Let $L=\prod_{p\in V\cup U}p$. Since $B\subseteq V$, we have $b\mid L$, so every $1/(uv)$ as well as $\tau$ is an integer multiple of $1/L$. For $S\subseteq A$ write $\sigma(S)=\sum_{n\in S}1/n$, and let $N$ be the number of $S\subseteq A$ with $\sigma(S)\equiv\tau\pmod1$. We must show $N>0$.

We write $e(x)=e^{2\pi i x}$; this function has period $1$. From now on, $h$ always runs over the integers with $-L/2<h\le L/2$, and $\sum_h$ denotes the sum over them; these $h$ form a complete set of residues modulo $L$, and we call them \emph{frequencies}. For an integer $m$, the geometric sum $\frac1L\sum_he(hm/L)$ equals $1$ if $L\mid m$ and $0$ otherwise. For each $S$, the number $m=L(\sigma(S)-\tau)$ is an integer, and it is divisible by $L$ exactly when $\sigma(S)\equiv\tau\pmod1$. So summing the geometric sum over all $S\subseteq A$ counts $N$:
\[
 N=\sum_{S\subseteq A}\frac1L\sum_he\big(h(\sigma(S)-\tau)\big)=\frac1L\sum_{h}e(-h\tau)\sum_{S\subseteq A}e\big(h\sigma(S)\big).
\]
Since $e(h\sigma(S))=\prod_{n\in S}e(h/n)$, the inner sum over all subsets factorises over the edges:
\[
 \sum_{S\subseteq A}e\big(h\sigma(S)\big)=\prod_{(v,u)\in A}\Big(1+e\Big(\frac h{uv}\Big)\Big).
\]
Dividing by the number $2^{|A|}$ of subsets and putting
\[
 \phi(h)=\prod_{(v,u)\in A}\frac{1+e(h/(uv))}{2}\qquad(h\in\Z),
\]
we obtain
\begin{equation}\label{eq:N}
 \frac{N}{2^{|A|}}=\frac1{L}\sum_{h}e(-h\tau)\,\phi(h).
\end{equation}
Each factor of $\phi(h)$ has absolute value at most $1$, and $\phi(0)=1$. So the frequency $h=0$ alone contributes $1/L$ to the right-hand side. The rest of the proof sorts all frequencies into three cases (\S\ref{sec:table}) and shows that one case contributes at least $1/L$, while the other two together contribute at most $1/(2L)$ in absolute value.

\section{The table and the three cases}\label{sec:table}

\subsection{The table}
Let $G_V=\prod_{v\in V}\Z/v\Z$ and $G_U=\prod_{u\in U}\Z/u\Z$. By the Chinese remainder theorem, the map $h\mapsto(\xi,\zeta)$, with $\xi=(h\bmod v)_{v\in V}\in G_V$ and $\zeta=(h\bmod u)_{u\in U}\in G_U$, is a bijection from our frequencies onto $G_V\times G_U$; we write $h\leftrightarrow(\xi,\zeta)$. Consider the phase
\[
 \varphi_{vu}(h):=\frac h{uv}\bmod1 .
\]
It depends only on $h$ modulo $uv$, hence, again by the Chinese remainder theorem, only on the pair $(\xi_v,\zeta_u)$: it equals $J/(uv)$, where $J$ is any integer $\equiv\xi_v\pmod v$ and $\equiv\zeta_u\pmod u$. We write $\varphi_{vu}(\xi,\zeta_u)$ for this common value.

For real $x$ we have
\begin{equation}\label{eq:half}
 1+e(x)=e\Big(\frac x2\Big)\Big(e\Big(-\frac x2\Big)+e\Big(\frac x2\Big)\Big)=2e\Big(\frac x2\Big)\cos\pi x ,
\end{equation}
so $|1+e(x)|=2|\cos\pi x|$ and $|\phi(h)|=\prod_A|\cos\pi\varphi_{vu}(h)|$. We picture this as a table: a $V\times U$ matrix whose $(v,u)$ entry is $|\cos\pi\varphi_{vu}|$. The entry depends only on the \emph{row label} $\xi_v$ and the \emph{column label} $\zeta_u$, it is at most $1$, and $|\phi(h)|$ is the product of all entries. For $u\in U$, $\xi\in G_V$ and $\zeta_u\in\Z/u\Z$, put
\[
 F_u(\xi,\zeta_u)=\prod_{v\in V}|\cos\pi\varphi_{vu}|,\qquad S_u(\xi)=\sum_{\zeta_u\in\Z/u\Z}F_u(\xi,\zeta_u).
\]
So $F_u$ is the product of the column $u$, the \emph{weight} of that column; it is at most $1$. Once all row labels $\xi$ are fixed, each column depends only on its own label $\zeta_u$, so the sum over $\zeta$ splits column by column:
\begin{equation}\label{eq:factor}
 \sum_{h}|\phi(h)|=\sum_{\xi\in G_V}\sum_{\zeta\in G_U}\prod_{u\in U}F_u(\xi,\zeta_u)=\sum_{\xi\in G_V}\ \prod_{u\in U}S_u(\xi).
\end{equation}
Trivially $S_u(\xi)\le u$. The main work is to show that $S_u(\xi)$ is in fact at most about $1$, and slightly less than $1$ unless the row labels have a very special form.

\subsection{The three cases}
We call the row labels $\xi\in G_V$ \emph{coherent} if there is an integer $M$ with $|M|\le y^7$ and $\xi_v\equiv M\pmod v$ for all $v\in V$, and \emph{incoherent} otherwise. Note that $\prod_Vv\ge\prod_{w\in W}w>(y^2)^{|W|}\ge y^{y^2/(4\log y)}$ by \eqref{eq:basic}, which exceeds $2y^7+1$ for large $y$; in particular $L>2y^7+1$. For coherent $\xi$ the integer $M$ is unique: two such integers are congruent modulo every $v\in V$, hence modulo $\prod_Vv$, and they differ by at most $2y^7<\prod_Vv$. We then write $\xi\equiv M$, and we call the column label $\zeta_u$ \emph{matching} (it matches $M$) if $\zeta_u\equiv M\pmod u$ and \emph{mismatched} otherwise.

Every frequency $h\leftrightarrow(\xi,\zeta)$ falls into exactly one of three cases.
\begin{center}
\renewcommand{\arraystretch}{1.25}
\begin{tabular}{c|c|c|p{0.36\textwidth}}
 case & rows $\xi$ & columns $\zeta$ & what the table looks like\\\hline
 I & coherent, $\xi\equiv M$ & all matching & all entries $\cos(\pi M/uv)$, positive and close to $1$\\
 II & coherent, $\xi\equiv M$ & some mismatched & a mismatched column is far from $1$ in many test rows\\
 III & incoherent & arbitrary & every column has an entry not extremely close to $1$
\end{tabular}
\end{center}
In case~I, $h\equiv M$ modulo every prime of the graph, i.e.\ $h\equiv M\pmod L$; since $|M|\le y^7<L/2$, this means $h=M$. Conversely every $h=M$ with $|M|\le y^7$ is in case~I. So case~I consists exactly of the frequencies $h$ with $|h|\le y^7$; these are the \emph{major arcs}, and cases~II and~III are the \emph{minor arcs}.

The plan is as follows.
\begin{itemize}
\item Case~I contributes at least $1/L$ (Proposition~\ref{prop:major}, \S\ref{sec:major}). This uses only that the angles $M/(uv)$ are small and that $\mu$ is close to $\tau$.
\item Two lemmas about a single column (\S\ref{sec:tools}). \emph{Separation}: two different labels of one column cannot both give it a non-negligible weight; and if one label has small phases in all test rows, every other label gives negligible weight. \emph{No wrap-around}: if some label makes all entries of a column close to $1$, then the row labels are coherent.
\item Case~II: every mismatched column has negligible weight, by separation. Case~III: every column loses a small factor, by no wrap-around and separation. Together cases~II and~III contribute at most $1/(2L)$ (Proposition~\ref{prop:minor}, \S\ref{sec:minor}).
\end{itemize}

\section{Case I: the major arcs}\label{sec:major}

On the major arcs every conjugate pair has a nonnegative contribution: all the angles $M/(uv)$ are small, and $\mu$ is within $H_V/(2y^8)$ of $\tau$. (The cutoff $y^7$ in the definition of coherence could be anywhere between $y^7/4$, which is what Lemma~\ref{lem:admissible} needs, and $y^8/(2H_V)$, which is what this section needs; $y^7$ is a convenient choice.)

\begin{proposition}\label{prop:major}
Let $\Sigma_{\rm maj}=\frac1L\sum_{|M|\le y^7}e(-M\tau)\phi(M)$ be the part of the right-hand side of \eqref{eq:N} coming from case~I. Then $\Sigma_{\rm maj}$ is real and $\Sigma_{\rm maj}\ge1/L$.
\end{proposition}
\begin{proof}
\emph{Step 1: the phase.} Applying \eqref{eq:half} to every factor of $\phi(M)$ and collecting the phases,
\[
 \phi(M)=\prod_{(v,u)\in A}e\Big(\frac M{2uv}\Big)\cos\frac{\pi M}{uv}=e(M\mu)\,P(M),\qquad P(M)=\prod_{(v,u)\in A}\cos\frac{\pi M}{uv},
\]
because $\sum_{(v,u)\in A}\frac1{2uv}=\mu$. Hence $e(-M\tau)\phi(M)=e(M\beta)P(M)$, where $\beta=\mu-\tau$.

\emph{Step 2: $P(M)>0$.} For $|M|\le y^7$, $u>y^8$ and $v\ge2$ we have $|M|/(uv)\le y^7/(2y^8)=1/(2y)<\frac12$. So every angle $\pi M/(uv)$ lies in $(-\frac\pi2,\frac\pi2)$, every factor of $P(M)$ is positive, and $P(M)>0$. Moreover $P(0)=1$, and $P(-M)=P(M)$ since cosine is even.

\emph{Step 3: pairing.} For $1\le M\le y^7$, the terms for $M$ and $-M$ add up to
\[
 \big(e(M\beta)+e(-M\beta)\big)P(M)=2\cos(2\pi M\beta)\,P(M).
\]
Therefore
\[
 \Sigma_{\rm maj}=\frac1L\Big(1+2\sum_{M=1}^{y^7}\cos(2\pi M\beta)\,P(M)\Big),
\]
which is real.

\emph{Step 4: signs.} By Lemma~\ref{lem:tuning}, $|\beta|\le H_V/(2y^8)$. So for $1\le M\le y^7$ we get $|2\pi M\beta|\le\pi H_V/y<\frac\pi2$ once $y>2H_V$, and hence $\cos(2\pi M\beta)>0$. Every term in the last sum is positive, so $\Sigma_{\rm maj}\ge1/L$.
\end{proof}

Both uses of smallness are visible here: the angles $M/(uv)$ are small because $|M|\le y^7$ is much smaller than every $u$, and the phases $M\beta$ are small because the tuning made $|\beta|$ of size $y^{-8}$.

\section{Two lemmas about a single column}\label{sec:tools}

We first fix some notation, used only from here on. For real $x$, $\nint{x}\in[0,\frac12]$ denotes the distance from $x$ to the nearest integer. For a modulus $m$, the \emph{balanced representative} of a residue class modulo $m$ is its representative in $(-m/2,m/2]$. For a prime $u$ and an integer $m$ coprime to $u$, $\bar m$ denotes the inverse of $m$ modulo $u$; below $m$ is always a small-side prime $v$ or $w$. The following elementary bound converts ``a phase is far from an integer'' into ``a table entry is small''.

\begin{lemma}\label{lem:cos}
For all real $x$ we have $|\cos\pi x|\le\exp(-2\nint{x}^2)$.
\end{lemma}
\begin{proof}
Let $t=\nint{x}\in[0,\frac12]$. Since $|\cos\pi x|$ has period $1$ and is even, $|\cos\pi x|=\cos\pi t$. The function $\sin\pi t$ is concave on $[0,\frac12]$ and equals $0$ and $1$ at the endpoints, so it lies above the chord: $\sin\pi t\ge2t$. Hence $\cos^2\pi t=1-\sin^2\pi t\le1-4t^2\le e^{-4t^2}$, using $1-z\le e^{-z}$. Taking square roots gives the claim.
\end{proof}

The key question of this section is the following. Fix the row labels $\xi$ and a column $u$. The label of this column is a residue $\zeta_u\in\Z/u\Z$, so it can take $u$ different values; changing it from $\zeta_u'$ to $\zeta_u$ means passing between two frequencies that agree modulo every prime of the graph except $u$. Can two different values $\zeta_u\ne\zeta_u'$ both give this column a large weight? To compare them we need to know how the phases of a column change when the rows are kept fixed and only the column label changes. If $h\equiv h'\pmod v$ and $h-h'\equiv s\pmod u$, then
\begin{equation}\label{eq:pf}
 \varphi_{vu}(h)-\varphi_{vu}(h')\equiv\frac{s\bar v}u\pmod1,
\end{equation}
because $h-h'=vk$ for an integer $k$ with $vk\equiv s\pmod u$, i.e.\ $k\equiv s\bar v\pmod u$, and then $(h-h')/(uv)=k/u$.

\subsection{Separation}
Put
\[
 \delta=\frac{1}{400\log y},\qquad \varepsilon=\exp\Big(-\frac{|W|\,\delta^2}{8}\Big)\le\exp\Big(-\frac{y^2}{2\cdot10^7\log^3y}\Big).
\]
Here $\delta$ is the threshold below which a phase counts as ``almost an integer'', and $\varepsilon$, which is smaller than any negative power of $y$, is the size below which a column weight counts as negligible; the bound for $\varepsilon$ follows from \eqref{eq:basic}.

Now take two frequencies $h,h'$ with the same row labels $\xi$ that differ only at the column $u$, say $h-h'\equiv s\not\equiv0\pmod u$. By \eqref{eq:pf}, their phases in row $w$ differ by
\[
 \varphi_{wu}(h)-\varphi_{wu}(h')\equiv\frac{s\bar w}{u}\pmod1 .
\]
Two points are worth stressing. First, this difference does not depend on the row labels at all: the rows are the same for $h$ and $h'$, so $\xi_w$ cancels, and the difference depends only on $s$, $u$ and the prime $w$ itself. Second, we want the difference to be clearly visible, i.e.\ $\nint{s\bar w/u}>\delta$, in many rows. Call $w$ \emph{invisible} if $\nint{s\bar w/u}\le\delta$. For such $w$ there is a small integer $t$, $|t|\le\delta u$, with $tw\equiv s\pmod u$, i.e.\ $tw=s+u\ell$, and $\ell$ takes only about $4\delta y^2$ values. So every invisible $w$ divides one of these few nonzero integers $s+u\ell$, each of which has at most $5$ prime factors exceeding $y^2$. Hence there are few invisible $w$: at most half of $W$, by the next lemma. In the other half of the rows the two phases differ by more than $\delta$, so they cannot both be close to an integer, whatever the row labels are. This is what Lemma~\ref{lem:unique} exploits.

\begin{lemma}\label{lem:count}
Let $u\in U$ and $s\not\equiv0\pmod u$. If $y$ is large, then
\[
 \#\{w\in W:\ \nint{s\bar w/u}\le\delta\}\le 5(4\delta y^2+2)\le |W|/2 .
\]
\end{lemma}
\begin{proof}
Take $s$ balanced, i.e.\ $|s|\le u/2$. For a counted $w$, let $t$ be the balanced representative of $s\bar w$ modulo $u$; then $\nint{s\bar w/u}=|t|/u$, so $|t|\le\delta u$. Since $tw\equiv s\pmod u$, we can write $tw=s+u\ell$ with an integer $\ell$. From $|t|\le\delta u$, $w\le2y^2$ and $|s|\le u/2$ we get $|u\ell|\le2\delta y^2u+u/2$, i.e.\ $|\ell|\le2\delta y^2+\frac12$. So there are at most $4\delta y^2+2$ possible values of $\ell$.

Fix such an $\ell$. The integer $s+u\ell$ is nonzero, because $u\nmid s$. Its absolute value is at most $u(2\delta y^2+1)\le y^9\cdot y^2=y^{11}$ for large $y$. And it is divisible by every counted $w$ belonging to this $\ell$. A product of six primes exceeding $y^2$ exceeds $y^{12}$, so a nonzero integer of absolute value at most $y^{11}$ has at most $5$ prime factors exceeding $y^2$. Hence each $\ell$ accounts for at most $5$ values of $w$, which gives the first inequality. For the second, $5(4\delta y^2+2)=y^2/(20\log y)+10$, and this is at most $y^2/(16\log y)\le|W|/2$ for large $y$, by \eqref{eq:basic}.
\end{proof}

Consequently two different labels of one column are far apart in many test rows, so they cannot both give the column a large weight. Call a label $\zeta_u$ \emph{heavy} (for given $\xi$) if $F_u(\xi,\zeta_u)>\varepsilon$.

\begin{lemma}[Separation]\label{lem:unique}
Let $u\in U$, $\xi\in G_V$ and $\zeta_u\not\equiv\zeta_u'\pmod u$.
\begin{enumerate}[label=(\roman*)]
\item $F_u(\xi,\zeta_u)F_u(\xi,\zeta_u')\le\varepsilon^2$. In particular, at most one label $\zeta_u$ is heavy.
\item If $\nint{\varphi_{wu}(\xi,\zeta'_u)}\le\delta/2$ for all $w\in W$, then $F_u(\xi,\zeta_u)\le\varepsilon$.
\end{enumerate}
\end{lemma}
\begin{proof}
Let $s=\zeta_u-\zeta_u'\not\equiv0\pmod u$. By Lemma~\ref{lem:count}, at least $|W|/2$ primes $w\in W$ have $\nint{s\bar w/u}>\delta$. For each such $w$, \eqref{eq:pf} says that the two phases $\varphi_{wu}(\xi,\zeta_u)$ and $\varphi_{wu}(\xi,\zeta'_u)$ differ by $s\bar w/u$ modulo $1$, so by the triangle inequality for $\nint{\cdot}$,
\[
 \nint{\varphi_{wu}(\xi,\zeta_u)}+\nint{\varphi_{wu}(\xi,\zeta'_u)}\ge\nint{s\bar w/u}>\delta .
\]
So one of the two phases exceeds $\delta/2$, and by Lemma~\ref{lem:cos} the corresponding cosine factor is at most $e^{-2(\delta/2)^2}=e^{-\delta^2/2}$. All other cosine factors are at most $1$. Multiplying over these $w$ gives
\[
 F_u(\xi,\zeta_u)F_u(\xi,\zeta_u')\le\exp\Big(-\frac{|W|}2\cdot\frac{\delta^2}2\Big)=\varepsilon^2,
\]
which is (i). If two labels were heavy, their product would exceed $\varepsilon^2$; so at most one is heavy. Under the hypothesis of (ii), the phase of $\zeta_u'$ at every $w$ is at most $\delta/2$, so for each of the $w$ above the large phase is the one of $\zeta_u$. The same computation, now for $F_u(\xi,\zeta_u)$ alone, gives $F_u(\xi,\zeta_u)\le\varepsilon^2\le\varepsilon$.
\end{proof}

\subsection{No wrap-around}
The next lemma says when a column can have all its entries close to $1$. It is the Chinese remainder theorem without wrap-around: small integers that are congruent modulo $u$ must be equal. Put
\[
 \delta_a=\frac1{8y^2}.
\]

In words, the lemma below says: for a column to be almost perfect, i.e.\ for all its entries to be very close to $1$ (every phase at most $\delta_a$), the row labels must be coherent. Case~III uses the contrapositive: if the rows are incoherent, then for every column and every label some entry has phase larger than $\delta_a$, so by Lemma~\ref{lem:cos} the column weight is at most
\[
 e^{-2\delta_a^2}=e^{-1/(32y^4)}\approx1-\frac1{32y^4}.
\]
This loss is tiny. Incoherent rows may well give a column a fairly large weight; the lemma only guarantees that the weight falls short of $1$ by about $y^{-4}$. It becomes effective only together with two other facts. By separation (Lemma~\ref{lem:unique}), at most one label per column is heavy, so $S_u(\xi)\le e^{-2\delta_a^2}+Y\varepsilon\le e^{-\delta_a^2}$ for large $y$, the light labels being absorbed (\S\ref{sec:minor}). And there are $|U|\ge\tau y^8/H_V$ columns, so these small losses multiply to $\exp(-\tau y^4/(64H_V))$, which beats the at most $2b\cdot16^{y^2}$ row patterns (\S\ref{sec:minor}). Note that the two lemmas of this section use very different thresholds: separation concerns weights above $\varepsilon$, a very weak requirement, whereas the no-wrap-around lemma concerns weights above $e^{-2\delta_a^2}$, i.e.\ columns that are almost perfect.

\begin{lemma}[No wrap-around]\label{lem:admissible}
Let $u\in U$, $\xi\in G_V$ and $\zeta_u\in\Z/u\Z$. If $\nint{\varphi_{vu}(\xi,\zeta_u)}\le\delta_a$ for all $v\in V$, then there is an integer $M$ with $|M|\le y^7/4$ and $\xi_v\equiv M\pmod v$ for all $v\in V$. In particular, $\xi$ is coherent.
\end{lemma}
\begin{proof}
For each $v$, let $J_v$ be the balanced representative modulo $uv$ of the integers that are $\equiv\xi_v\pmod v$ and $\equiv\zeta_u\pmod u$. Then $\varphi_{vu}=J_v/(uv)\bmod1$ with $|J_v|\le uv/2$, so $\nint{\varphi_{vu}}=|J_v|/(uv)$, and the hypothesis gives $|J_v|\le\delta_auv\le u/4$, since $v\le2y^2$. All the $J_v$ are $\equiv\zeta_u\pmod u$, so any two of them differ by a multiple of $u$; but they differ by at most $u/4+u/4<u$, so they are all equal to one integer $M$. Then $M=J_v\equiv\xi_v\pmod v$ for every $v$. Finally, the case $v=2$ gives $|M|=|J_2|\le2\delta_au\le Y/(4y^2)\le y^7/4$.
\end{proof}

\section{Cases II and III: the minor arcs}\label{sec:minor}

\begin{proposition}\label{prop:minor}
For all sufficiently large $y$ (in terms of $b$),
\[
 \sum_{h\text{ in cases II and III}}|\phi(h)|\le\frac12 .
\]
\end{proposition}
\begin{proof}
We use \eqref{eq:factor} and show that each of the two cases contributes at most $\frac14$.

\emph{Case III (incoherent rows): a small loss in every column.} Fix an incoherent $\xi$ and a column $u$. By Lemma~\ref{lem:admissible}, every label $\zeta_u$ has some $v$ with $\nint{\varphi_{vu}(\xi,\zeta_u)}>\delta_a$: otherwise $\xi$ would be coherent. So by Lemma~\ref{lem:cos} every label has $F_u(\xi,\zeta_u)\le e^{-2\delta_a^2}$. By Lemma~\ref{lem:unique}(i), at most one label is heavy, and the other at most $u\le Y$ labels have $F_u(\xi,\zeta_u)\le\varepsilon$. Hence
\[
 S_u(\xi)\le e^{-2\delta_a^2}+Y\varepsilon\le e^{-\delta_a^2}.
\]
The last inequality holds for large $y$, because $e^{-\delta_a^2}-e^{-2\delta_a^2}\ge\delta_a^2/2=1/(128y^4)$, while $Y\varepsilon\le y^9\varepsilon$ is much smaller. So each of the $|U|$ columns loses a factor $e^{-\delta_a^2}$. Summing over all incoherent $\xi$, of which there are at most $|G_V|=\prod_Vv$, and using \eqref{eq:basic} and $|U|\ge\tau y^8/H_V$ from Lemma~\ref{lem:tuning}, the total is at most
\[
 2b\cdot16^{y^2}\exp(-\delta_a^2|U|)\le2b\cdot16^{y^2}\exp\Big(-\frac{\tau y^4}{64H_V}\Big)\le\frac14
\]
for large $y$, since $16^{y^2}=e^{y^2\log16}$ grows much more slowly than $e^{cy^4}$.

\emph{Case II (coherent rows, some mismatched column): a huge loss in each mismatched column.} Fix a coherent $\xi\equiv M$. The matching label $M\bmod u$ has small phases in the test rows: the integer $J=M$ is $\equiv\xi_w\pmod w$ and $\equiv M\pmod u$, so for $w\in W$,
\[
 \nint{\varphi_{wu}(\xi,M\bmod u)}\le\frac{|M|}{uw}\le\frac{y^7}{y^8\cdot y^2}=y^{-3}\le\frac\delta2
\]
for large $y$. By Lemma~\ref{lem:unique}(ii), every mismatched label $\zeta_u\not\equiv M$ has $F_u(\xi,\zeta_u)\le\varepsilon$. So we can write $S_u(\xi)=F_u(\xi,M\bmod u)+R_u$, where $R_u$, the sum over the fewer than $Y$ mismatched labels, satisfies $R_u\le Y\varepsilon$.

Now expand the product:
\[
 \prod_{u\in U}S_u(\xi)=\sum_{D\subseteq U}\ \prod_{u\notin D}F_u(\xi,M\bmod u)\prod_{u\in D}R_u ,
\]
where $D$ is the set of mismatched columns. The term $D=\emptyset$ is the single frequency $h=M$ of case~I. Since every $F_u\le1$, the remaining terms, which are exactly the frequencies of case~II with $\xi\equiv M$, contribute at most
\[
 \sum_{\emptyset\ne D\subseteq U}(Y\varepsilon)^{|D|}=(1+Y\varepsilon)^{|U|}-1\le|U|Y\varepsilon\,e^{|U|Y\varepsilon},
\]
using $(1+z)^n-1\le e^{nz}-1\le nz\,e^{nz}$. Here $|U|,Y\le y^9$. Summing over the $2y^7+1\le3y^7$ possible values of $M$, case~II contributes at most $3y^{25}\varepsilon\,e^{y^{18}\varepsilon}$, which is at most $\frac14$ for large $y$ because $y^k\varepsilon\to0$ for every $k$.
\end{proof}

\begin{proof}[Proof of Proposition~\ref{prop:core}]
Split the right-hand side of \eqref{eq:N} as $\Sigma_{\rm maj}+\Sigma_{\rm min}$ according to the three cases. By Proposition~\ref{prop:major}, $\Sigma_{\rm maj}$ is real and at least $1/L$. By Proposition~\ref{prop:minor}, $|\Sigma_{\rm min}|\le\frac1L\sum_{|h|>y^7}|\phi(h)|\le\frac1{2L}$. The left-hand side $N/2^{|A|}$ is real, so taking real parts, for $y$ large,
\[
 \frac{N}{2^{|A|}}=\Sigma_{\rm maj}+\operatorname{Re}\Sigma_{\rm min}\ \ge\ \frac1L-\frac1{2L}>0 .
\]
Thus $N>0$, i.e.\ some $S\subseteq A$ has $\sum_{n\in S}1/n\equiv\tau\pmod1$.
\end{proof}

This completes the proof of Theorem~\ref{thm:small}, and hence of Theorem~\ref{thm:main}.

\section*{Use of AI}
This work was carried out as a human--AI collaboration. The author directed the mathematics and is responsible for all content. AI assistants were used throughout the project: to explore constructions and proof strategies, to write and run programs for the numerical experiments that guided the construction, to check and simplify the arguments through critical reviews, to write the Lean formalisation, and to draft and revise the exposition. Precisely because the development was AI-assisted, the proof has been formalised in Lean~4 with Mathlib (see ``Formal verification'' in \S1): all proof steps, including those written with AI assistance, are checked by Lean's kernel, conditional only on the stated inequality of Ramanujan. It uses only elementary inputs, and it is also short enough to be checked by hand.

\section*{Acknowledgements}
We thank Yuren Tang, whose Lean development \cite{Tan26} was the first to show that the circle method works for this problem and provided its basic framework. The present proof adopts that framework; what we add is a different construction that simplifies the treatment of the minor arcs and of the main term (see ``Relation to previous work'' in \S1).

\end{document}